\documentclass[11pt,reqno]{amsart}

\usepackage{amssymb, amsmath, amsthm}
\usepackage{hyperref}
\usepackage[alphabetic,lite]{amsrefs}
\usepackage{verbatim}
\usepackage{amscd}   
\usepackage[all]{xy} 
\usepackage{youngtab} 
\usepackage{young} 
\usepackage{tikz}
\usepackage{ mathrsfs }

\newcommand{\defi}[1]{{\upshape\sffamily #1}}

\renewcommand{\a}{\alpha}
\renewcommand{\b}{\beta}
\renewcommand{\c}{\gamma}
\newcommand{\bw}{\bigwedge}
\newcommand{\K}{\bb{K}}
\renewcommand{\ll}{\lambda}
\newcommand{\m}{\mathfrak{m}}

\newcommand{\oo}{\otimes}

\newcommand{\rank}{\textrm{rank}}

\newcommand{\Ext}{\operatorname{Ext}}
\newcommand{\GL}{\operatorname{GL}}
\newcommand{\Hom}{\operatorname{Hom}}

\newcommand{\SL}{\operatorname{SL}}

\newcommand{\Sym}{\operatorname{Sym}}
\newcommand{\Tor}{\operatorname{Tor}}

\newcommand{\bb}[1]{\mathbb{#1}}
\renewcommand{\rm}[1]{\textrm{#1}}
\newcommand{\mc}[1]{\mathcal{#1}}

\newcommand{\tl}[1]{\tilde{#1}}

\def\lra{\longrightarrow}

\newtheorem{theorem}{Theorem}[section]
\newtheorem*{theorem*}{Theorem}
\newtheorem{lemma}[theorem]{Lemma}

\newtheorem{corollary}[theorem]{Corollary}
\newtheorem*{corollary*}{Corollary}

\newtheorem*{covariants*}{Theorem on Covariants of the Special Linear Group}
\newtheorem*{minors*}{Theorem on Maximal Minors}
\newtheorem*{pfaffians*}{Theorem on sub--maximal Pfaffians}
\newtheorem*{ext*}{Theorem on Ext modules}

\theoremstyle{definition}

\newtheorem*{definition*}{Definition}
\newtheorem{example}[theorem]{Example}

\theoremstyle{remark}
\newtheorem{remark}[theorem]{Remark}
\newtheorem*{remark*}{Remark}

\numberwithin{equation}{section}

\begin{document}

\title{Local cohomology with support in ideals of maximal minors and sub--maximal Pfaffians}

\author{Claudiu Raicu}
\address{Department of Mathematics, Princeton University, Princeton, NJ 08544\newline
\indent Institute of Mathematics ``Simion Stoilow'' of the Romanian Academy}
\email{craicu@math.princeton.edu}

\author{Jerzy Weyman}
\address{Department of Mathematics, Northeastern University, Boston, MA 02115}
\email{j.weyman@neu.edu}

\author{Emily E. Witt}
\address{Department of Mathematics, University of Minnesota, Minneapolis, MN 55455}
\email{ewitt@umn.edu}


\date{\today}


\begin{abstract} We compute the $\GL$--equivariant description of the local cohomology modules with support in the ideal of maximal minors of a generic matrix, as well as of those with support in the ideal of $2n\times 2n$ Pfaffians of a $(2n+1)\times(2n+1)$ generic skew--symmetric matrix. As an application, we characterize the Cohen--Macaulay modules of covariants for the action of the special linear group $\SL(G)$ on $G^{\oplus m}$. The main tool we develop is a method for computing certain Ext modules based on the geometric technique for computing syzygies and on Matlis duality.
\end{abstract}

\maketitle

\section{Introduction}\label{sec:intro}

In this paper we present the $\GL$--equivariant description of the local cohomology modules of a polynomial ring $S$ with support in an ideal $I$ (denoted $H^j_I(S)$ for $j\geq 0$), in two cases of interest (throughout the paper, $\K$ will denote a field of characteristic zero):
\begin{itemize}
 \item $S$ is the ring of polynomial functions on the vector space of $m\times n$ matrices with entries in $\K$, and $I$ is the ideal of $S$ generated by the polynomial functions that compute the $n\times n$ minors.
 
 \item $S$ is the ring of polynomial functions on the vector space of $(2n+1)\times (2n+1)$ skew--symmetric matrices with entries in $\K$, and $I$ is the ideal generated by the polynomial functions that compute the $2n\times 2n$ Pfaffians.
\end{itemize}

One of the motivations behind our investigation is trying to understand the Cohen--Macaulayness of modules of covariants. This problem has a long history, originating in the work of Stanley \cite{stanley} on solution sets of linear Diophantine equations (see \cite{VdB:survey} for a survey, and also \cites{brion,VdB:CMcov,VdB:sl2,VdB:loccoh}). When $H$ is a reductive group and $W$ a finite dimensional $H$--representation, a celebrated theorem of Hochster and Roberts \cite{hochster-roberts} asserts that the ring of invariants $S^H$, with respect to the natural action of $H$ on the polynomial ring $S=\Sym(W)$, is Cohen--Macaulay. If $U$ is another finite dimensional $H$--representation, the associated \defi{module of covariants} is defined as $(S\oo U)^H$, and is a finitely generated $S^H$--module. In general it is quite rare that $(S\oo U)^H$ is Cohen--Macaulay, and our first result illustrates this in a special situation. 

\begin{covariants*}[{Theorem \ref{thm:covariants}}]
 Consider a finite dimensional $\K$--vector space $G$ of dimension $n$, an integer $m>n$, and let $H=\SL(G)$ be the special linear group, $W=G^{\oplus m}$, and $S=\Sym(W)$. If $U=S_{\mu}G$ is the irreducible $H$--representation associated to the partition $\mu=(\mu_1\geq\mu_2\geq\cdots\geq\mu_n=0)$ then $(S\oo U)^H$ is Cohen--Macaulay if and only if $\mu_s-\mu_{s+1}<m-n$ for all $s=1,\cdots,n-1$. 
\end{covariants*}

In the case when $m=n+1$, the Theorem on Covariants asserts that the only Cohen--Macaulay modules of covariants are direct sums of copies of $S^H$, which is remarked at the end of \cite{brion}. The case $n=3$ of the theorem is explained in \cite{VdB:loccoh}, while the case $n=2$ for an arbitrary $H$--representation $W$ is treated in \cite{VdB:sl2}. We restrict to $m>n$ in the statement of the theorem to avoid trivial cases: if $m<n$ then $S^H=\K$, while for $m=n$, $S^H=\K[\det]$ is a polynomial ring in one variable $\det$, corresponding to the determinant of the generic $n\times n$ matrix; in both cases, all the modules of covariants are Cohen--Macaulay.

We will prove the Theorem on Covariants by computing explicitly the local cohomology modules $H^j_I(S)$, where $I$ is the ideal generated by the maximal minors of a generic $m\times n$ matrix, and using the relationship between the local cohomology of the module $(S\oo U)^H$ and the invariants $(H^j_I(S)\oo U)^H$ \cite[Lemma~4.1]{VdB:tracerings}. The indices $j$ for which $H^j_I(S)\neq 0$ have been previously determined by the third author, as well as the description of the top non--vanishing local cohomology module $H^{n\cdot(m-n)+1}_I(S)$ as the injective hull of the residue field: see \cite{witt} for details, including some history behind the problem and its positive characteristic analogue. Writing $W=G^{\oplus m}=F\oo G$ for some $m$--dimensional vector space $F$, we note that the ideal $I$ (and hence the local cohomology modules $H^j_I(S)$) is preserved by the action of the product $\GL(F)\times\GL(G)$ of general linear groups. Our explicit description of the modules $H^j_I(S)$ exhibits their 
decomposition as a direct sum of irreducible representations of this group.

\begin{minors*}[{Theorem \ref{thm:loccohminors}}]
 Consider $\K$--vector spaces $F$ and $G$ of dimensions $m$ and $n$ respectively, with $m>n$, and integers $r\in\bb{Z}$, $j\geq 0$. For $1\leq s\leq n$ and $\ll=(\ll_1,\cdots,\ll_n)\in\bb{Z}^n$ a dominant weight (i.e. $\ll_1\geq\ll_2\geq\cdots\geq\ll_n$), let
\[\ll(s)=(\ll_1,\cdots,\ll_{n-s},\underbrace{-s,\cdots,-s}_{m-n},\ll_{n-s+1}+(m-n),\cdots,\ll_n+(m-n)).\]
We let $W(r;s)$ denote the set of dominant weights $\ll\in\bb{Z}^n$ with $\sum_{i=1}^n \ll_i=r$ and such that $\ll(s)\in\bb{Z}^m$ is also dominant. We identify $F^*\oo G^*$ with the vector space of $m\times n$ matrices, and write $S=\Sym(F\oo G)$ for the ring of polynomial functions on this space. We let $I$ denote the ideal generated by the (unique) irreducible $\GL(F)\times\GL(G)$--subrepresentation of $\Sym^n(F\oo G)$ isomorphic to $\bw^n F\oo\bw^n G$ (corresponding to the maximal minors of the generic $m\times n$ matrix). The degree $r$ part of the local cohomology module $H^j_I(S)$ decomposes as a direct sum of irreducible $\GL(F)\times\GL(G)$--representations as follows:
\[
H^j_I(S)_r=
\begin{cases}
\bigoplus_{\ll\in W(r;s)} S_{\ll(s)}F\oo S_{\ll}G, & \textrm{ if } j=s\cdot(m-n)+1,\textrm{ for some } s=1,\cdots,n; \\
0, & \textrm{otherwise}.  
\end{cases}
\]
\end{minors*}
In particular, if we write $S^*=\bigoplus_{i\geq 0}\Hom_{\K}(S_i,\K)$ for the graded dual of $S$, then the top non--vanishing local cohomology module $H^{n\cdot(m-n)+1}_I(S)$ is described as
\[H^{n\cdot(m-n)+1}_I(S)=\det (F^*\oo G^*)\oo S^*.\]

\begin{example}\label{ex:minors}
 Suppose that $m=3$ and $n=2$, so that we are dealing with the $2\times 2$ minors of a $3\times 2$ matrix. The only non--vanishing local cohomology modules are $H^2_I(S)$ and $H^3_I(S)$. The latter is just the injective hull of the residue field, as was first shown by Walther \cite[Example~6.1]{walther}. The Theorem on Maximal Minors yields the description of the former as
\[H^2_I(S)=\bigoplus_{a\geq -1\geq b}S_{(a,-1,b)}F\oo S_{(a,-1+b)}G,\]
where the graded components are determined by the value of $a+b-1$.
\end{example}

The corresponding result for sub--maximal Pfaffians is as follows:

\begin{pfaffians*}[{Theorem \ref{thm:loccohpfaffians}}]
 Consider a $\K$--vector space $F$ of dimension $2n+1$, and integers $r\in\bb{Z}$, $j\geq 0$. For $1\leq s\leq n$ and $\ll=(\ll_1,\cdots,\ll_{2n})\in\bb{Z}^{2n}$ a dominant weight, let
\[\ll(s)=(\ll_1,\cdots,\ll_{2n-2s},-2s,\ll_{2n-2s+1}+1,\cdots,\ll_{2n}+1).\]
We let $W(r;s)$ denote the set of dominant weights $\ll\in\bb{Z}^{2n}$ with $\sum_{i=1}^{2n} \ll_i=2r$, satisfying $\ll_{2i-1}=\ll_{2i}$ for $i=1,\cdots,n$, and such that $\ll(s)\in\bb{Z}^{2n+1}$ is also dominant. We identify $\bw^2 F^*$ with the vector space of $(2n+1)\times(2n+1)$ skew--symmetric matrices, and write $S=\Sym\left(\bw^2 F\right)$ for the ring of polynomial functions on this space. We let $I$ denote the ideal generated by the unique irreducible $\GL(F)$--subrepresentation of $\Sym^n\left(\bw^2 F\right)$ isomorphic to $\bw^{2n} F$ (corresponding to the sub--maximal Pfaffians of the generic $(2n+1)\times(2n+1)$ skew--symmetric matrix). The degree $r$ part of the local cohomology module $H^j_I(S)$ decomposes as a direct sum of irreducible $\GL(F)$--representations as follows:
\[
H^j_I(S)=
\begin{cases}
\bigoplus_{\ll\in W(r;s)} S_{\ll(s)}F, & \textrm{ if } j=2s+1,\textrm{ for some } s=1,\cdots,n; \\
0, & \textrm{otherwise}.  
\end{cases}
\]
\end{pfaffians*}
Similar to the maximal minors case, the top non--vanishing local cohomology module is described by
\[H^{2n+1}_I(S)=\det\left(\bw^2 F^*\right)\oo S^*.\]

\begin{example}\label{ex:pfaffians}
 Suppose that $n=2$, so that we are dealing with the $4\times 4$ Pfaffians of a $5\times 5$ skew--symmetric matrix. The only non--vanishing local cohomology modules are $H^3_I(S)$ and $H^5_I(S)$, and
\[H^3_I(S)=\bigoplus_{a\geq -2\geq b}S_{(a,a,-2,b,b)}F,\]
where the graded pieces are determined by the value of $a+b-1$.
\end{example}

Our strategy for proving the Theorems on Maximal Minors and on sub--maximal Pfaffians is to use the description of the local cohomology modules as a direct limit \cite[Appendix~1]{eis-syzygies}
\begin{equation}\label{eq:loccohlimitExt}
H^j_I(S)=\varinjlim_d\Ext^j_S(S/I^d,S), 
\end{equation}
and compute explicitly the relevant $\Ext$ modules. To do this, we develop a method for computing certain Ext modules, inspired by the geometric technique for computing syzygies \cite[Ch.~5]{weyman}:

\begin{ext*}[{Theorem \ref{thm:duality}}]
 Let $X$ be a projective variety, and let $W$ be a finite dimensional $\K$--vector space. Suppose that there is an exact sequence
\[0\lra \xi\lra W\oo\mc{O}_X\lra \eta\lra 0,\]
where $\xi$ and $\eta$ are vector bundles on $X$, and denote the rank of $\xi$ by $k$. Consider a vector bundle $\mc{V}$ on $X$ and define
\[\mc{M}(\mc{V})=\mc{V}\oo\Sym(\eta),\quad\mc{M}^*(\mc{V})=\mc{V}\oo\det(\xi)\oo\Sym(\eta^*).\]
Suppose that $H^j(X,\mc{M}(\mc{V}))=0$ for $j>0$, and let
\[M(\mc{V})=H^0(X,\mc{M}(\mc{V})).\] 
With the appropriate grading convention (see Section~\ref{sec:geomtechnique}), we have for each $j\geq 0$ a graded isomorphism
\begin{equation}\label{eq:ExtH}
\Ext^j_S(M(\mc{V}),S)=H^{k-j}(X,\mc{M}^*(\mc{V}))^*, 
\end{equation}
where $(-)^*$ stands for the graded dual.
\end{ext*}

When $I$ is the ideal of maximal minors of a generic matrix, respectively the ideal of sub--maximal Pfaffians of a generic skew--symmetric matrix, knowing $\Ext^*_S(S/I^d,S)$ is equivalent to knowing $\Ext^*_S(I^d,S)$, by a standard long exact sequence argument. In both cases we exhibit $I^d$ as $M(\mc{V}^d)$ for an appropriate choice of a line bundle $\mc{V}^d$ on a Grassmann variety. The explicit description of the $\Ext$ modules is then obtained from (\ref{eq:ExtH}) using Bott's Theorem. The final step in computing the local cohomology modules involves showing that the natural maps $\Ext^j_S(I^d,S)\to\Ext^j_S(I^{d+1},S)$ are injective, which is done by exhibiting the graded duals of these maps as the induced maps in cohomology for certain split surjections $\mc{M}^*(\mc{V}^{d+1})\twoheadrightarrow\mc{M}^*(\mc{V}^d)$.

In more general situations, such as that of the ideals of lower order minors, one can't expect a nice description of $I^d$ as the push-forward from a desingularization of a vector bundle with vanishing higher cohomology. However, for modules $M$ that admit a filtration with associated quotients $Q_i$ that have such a nice description, one can still use duality to compute $\Ext^*_S(Q_i,S)$ and then attempt to reconstruct $\Ext^*_S(M,S)$ from this data. This strategy is employed in \cite{raicu-weyman} in order to calculate all the modules $\Ext^*_S(S/I,S)$ when $I$ is generated by an irreducible $\GL(F)\times\GL(G)$-subrepresentation of $S=\Sym(F\oo G)$.

The paper is organized as follows. In Section~\ref{sec:prelim} we recall some basic facts regarding the representation theory of general linear groups in characteristic zero (Section~\ref{subsec:repthy}), and state Bott's theorem for computing the cohomology of homogeneous vector bundles on Grassmannians (Section~\ref{subsec:bott}). In Section~\ref{sec:geomtechnique} we prove the Theorem on Ext modules. In Section~\ref{sec:maxminors} we use this theorem to compute the local cohomology modules with support in the ideal of maximal minors of a generic matrix, and we prove the Theorem on Covariants. In Section~\ref{sec:maxpfaffians} we perform the corresponding calculation for the ideal of $2n\times 2n$ Pfaffians of a $(2n+1)\times(2n+1)$ generic skew--symmetric matrix.

\section{Preliminaries}\label{sec:prelim}

\subsection{Representation Theory {\cite{ful-har}, \cite[Ch.~2]{weyman}}}\label{subsec:repthy}
Throughout the paper, $\K$ will denote a field of characteristic $0$. If $W$ is a $\K$--vector space of dimension $\dim(W)=N$, a choice of basis determines an isomorphism between $\GL(W)$ and $\GL_N(\K)$. We will refer to $N$--tuples $\ll=(\ll_1,\cdots,\ll_N)\in\bb{Z}^N$ as \defi{weights} of the corresponding maximal torus of diagonal matrices. We say that $\ll$ is a \defi{dominant weight} if $\ll_1\geq\ll_2\geq\cdots\geq\ll_N$. Irreducible representations of $\GL(W)$ are in one-to-one correspondence with dominant weights $\ll$. We denote by $S_{\ll}W$ the irreducible representation associated to $\ll$, often referred to as a \defi{Schur functor}. We write $(a^N)$ for the weight with all parts equal to $a$, and define the \defi{determinant} of $W$ by $\det(W)=S_{(1^N)}W=\bw^N W$. We have $S_{\ll}W\oo\det(W)=S_{\ll+(1^N)}W$, and $S_{\ll}W^*=S_{(-\ll_N,\cdots,-\ll_1)}W$. We write $|\ll|$ for the total size $\ll_1+\cdots+\ll_N$ of $\ll$.
 
When $\ll$ is a dominant weight with $\ll_N\geq 0$, we say that $\ll$ is a \defi{partition} of $r=|\ll|$, and we write $\ll\vdash r$. Note that when we're dealing with partitions we often omit the trailing zeros, so $\ll=(5,2,1)\vdash 8$ is the same as $\ll=(5,2,1,0,0,0)$. The \defi{transpose} $\ll'$ of a partition $\ll$ is obtained by transposing the associated Young diagram. For $\ll=(5,2,1)$, $\ll'=(3,2,1,1,1)$. If $\mu$ is another partition, we write $\mu\subset\ll$ to indicate that $\mu_i\leq\ll_i$ for all $i$.

We will use freely throughout the paper the \defi{Cauchy formulas} which describe, given finite dimensional $\K$--vector spaces $F,G$, the decomposition of the symmetric and exterior powers of $F\oo G$ into a sum of irreducible $\GL(F)\times\GL(G)$--representations \cite[Cor.~2.3.3]{weyman}:
\begin{equation}\label{eq:cauchy}
 \Sym^r(F\oo G)=\bigoplus_{\ll\vdash r} S_{\ll}F\oo S_{\ll}G,\quad
  \bw^r(F\oo G)=\bigoplus_{\ll\vdash r} S_{\ll}F\oo S_{\ll'}G.
\end{equation}
In Section~\ref{sec:maxpfaffians} we will need the description of the symmetric powers of $\bw^2 W$ \cite[Ex.~I.8.6]{macdonald}:
\begin{equation}\label{eq:syme2}
 \Sym^r\left(\bw^2 W\right)=\bigoplus_{\substack{\ll\vdash 2r \\ \ll'_i \textrm{ even}}} S_{\ll}W.
\end{equation}
Note that the condition that $\ll'_i$ is even in (\ref{eq:syme2}) is equivalent to $\ll_1=\ll_2$, $\ll_3=\ll_4$, etc.

\subsection{Bott's Theorem for Grassmannians {\cite[Ch.~4]{weyman}}}\label{subsec:bott}

Consider a finite dimensional $\K$--vector space $W$ of dimension $N$, and a non-negative integer $k\leq N$. Consider the Grassmannian $\bb{G}=\bb{G}(k,W)$ of $k$--dimensional quotients of $W$. On $\bb{G}$ we have the tautological exact sequence
\begin{equation}\label{eq:tautological}
0\lra\mc{R}\lra W\oo\mc{O}_{\bb{G}}\lra\mc{Q}\lra 0, 
\end{equation}
where $\mc{R}$ (resp. $\mc{Q}$) denotes the universal sub--bundle (resp. quotient bundle). We have $\rank(\mc{Q})=k$ and $\rank(\mc{R})=N-k$. Consider two dominant weights $\a=(\a_1,\cdots,\a_k)$ and $\b=(\b_1,\cdots,\b_{N-k})$, and their concatenation $\c=(\c_1,\cdots,\c_r)$. Let $\delta=(N-1,\cdots,0)$, and consider $\c+\delta=(\c_1+N-1,\c_2+N-2,\cdots,\c_N)$. We write $\rm{sort}(\c+\delta)$ for the sequence obtained by arranging the entries of $\c+\delta$ in non-increasing order, and define 
\begin{equation}\label{eq:tildegamma}
\tl{\c}=\rm{sort}(\c+\delta)-\delta. 
\end{equation}

\begin{example}\label{ex:sortgamma}
 If $k=2$, $N=5$, $\a=(3,1)$ and $\b=(4,4,2)$ then $\c=(3,1,4,4,2)$, $\delta=(4,3,2,1,0)$, $\c+\delta=(7,4,6,5,2)$, $\rm{sort}(\c+\delta)=(7,6,5,4,2)$, and $\tl{\c}=(3,3,3,3,2)$.
\end{example}

\begin{theorem}[Bott]\label{thm:bott}
 With the above notation, if $\c+\delta$ has repeated entries, then
\[H^t(\bb{G},S_{\a}\mc{Q}\oo S_{\b}\mc{R})=0\ \forall t\geq 0.\]
Otherwise, writing $l$ for the number of pairs $(x,y)$ with $1\leq x<y\leq N$ and $\c_x-x<\c_y-y$, we have
\[H^l(\bb{G},S_{\a}\mc{Q}\oo S_{\b}\mc{R})=S_{\tl{\c}}W,\]
and $H^t(\bb{G},S_{\a}\mc{Q}\oo S_{\b}\mc{R})=0$ for $t\neq l$. As a consequence, $H^t(\bb{G},S_{\a}\mc{Q})=0$ for all $t>0$ and all $\a=(\a_1,\cdots,\a_k)$ with $\a_k\geq 0$.
\end{theorem}

\begin{example}\label{ex:bottcohom}
 Continuing with the notation from Example~\ref{ex:sortgamma}, we have $l=2$ ($-1=\c_2-2<\c_3-3=1$ and $-1=\c_2-2<\c_4-4=0$), so
\[H^2(S_{(3,1)}\mc{Q}\oo S_{(4,4,2)}\mc{R})=S_{(3,3,3,3,2)}W.\]
\end{example}

The following observation will be used in the proof of Theorem~\ref{thm:extIdminors}.

\begin{remark}\label{rem:alpharectangle}
 If all parts of $\b$ are equal, i.e. $\b_1=\cdots=\b_{N-k}$, then it follows that $H^t(S_{\a}\mc{Q}\oo S_{\b}\mc{R})=0$ if $t$ is not divisible by $(N-k)$. To see this, note that if $\c+\delta$ has distinct entries, then for $j\leq k$ we either have that $\c_j-j>\c_{k+1}-(k+1)$ or $\c_j-j<\c_N-N$. It follows that the pairs $(x<y)$ with $\c_x-x<\c_y-y$ come in groups of size $(N-k)$.
\end{remark}

\section{Ext modules via the geometric technique and Matlis duality}\label{sec:geomtechnique}

The goal of this section is to prove the Theorem on Ext modules. To achieve this, we combine the main idea behind the geometric technique for computing syzygies {\cite[Ch.~5]{weyman} with graded Matlis duality \cite[Sec.~I.3.6]{bru-her}}. We consider a projective variety $X$, a finite dimensional $\K$--vector space $W$ of dimension $\dim(W)=N$, and an exact sequence
\begin{equation}\label{eq:tautologicalV}
 0\lra \xi\lra W\lra \eta\lra 0,
\end{equation}
where $\xi$ and $\eta$ are vector bundles on $X$. In the above equation, and throughout the rest of paper, we abuse notation by writing $W$ for the coherent sheaf $\mc{O}_X\oo W$. We denote $\Sym(W)$ by $S$, and make the identification $S=\mc{O}_X\oo S$, thinking of $S$ as a sheaf of graded algebras on $X$. Writing $k$ for the rank of $\xi$, the exact sequence (\ref{eq:tautologicalV}) gives rise to a resolution of $\Sym(\eta)$ by graded $S$--modules:
\begin{equation}\label{eq:reswedgeeta}
0\lra\bw^k\xi\oo S(-k)\lra\cdots\lra\bw^i\xi\oo S(-i)\lra\cdots\lra\xi\oo S(-1)\lra S\lra\Sym(\eta)\lra 0.
\end{equation}
Given a vector bundle $\mc{V}$ on $X$, we define the vector bundle $\mc{M}(\mc{V})$ on $X$ by
\begin{equation}\label{eq:MV}
\mc{M}(\mc{V})=\mc{V}\oo\Sym(\eta).
\end{equation}
$\mc{M}(\mc{V})$ is naturally a graded $S$--module, with $\mc{M}_r(\mc{V}) = \mc{V}\oo\Sym^r(\eta)$, $r\geq 0$.

We consider the category of graded quasi--coherent $S$--modules on $X$ (equivalently, the category of quasi--coherent sheaves on $X\times W^*$, which are graded as sheaves on the affine space $W^*$; see \cite[Ex.~II.5.17]{hartshorne}). We write $D^b(X;W)$ for the corresponding bounded derived category. Tensoring (\ref{eq:reswedgeeta}) with $\mc{V}$, we get an identification in $D^b(X;W)$ between $\mc{M}(\mc{V})$ (considered as a complex concentrated in degree~$0$) and $\mc{F}_{\bullet}$, where $\mc{F}_{\bullet}$ is the complex resolving $\mc{M}(\mc{V})$:
\begin{equation}\label{eq:Fi}
\mc{F}_i=\mc{V}\oo\bw^i\xi\oo S(-i),
\end{equation}
with $\mc{F}_i$ living in homological degree $i$ (cohomological degree $-i$).

We write $M^*$ for the graded dual of a graded $S$--module $M$, in particular $S^*_r=\Sym^{-r}W^*$. Note that the graded Matlis dual $\hat{M}$ is related to the graded dual by a shift in degrees: $\hat{M}=M^*\oo\det(W^*)$, graded by $\hat{M}_r=(M^*)_{r+N}\oo\det(W^*)$. Taking the graded dual of (\ref{eq:reswedgeeta}) we obtain a resolution of $\Sym(\eta^*)$:
\begin{equation}\label{eq:reswedgeeta*}
0\lra\Sym(\eta^*)\lra S^*\lra \xi^*\oo S^*(1)\lra\cdots\lra\bw^k\xi^*\oo S^*(k)\lra 0.
\end{equation}
Define $\mc{M}^*(\mc{V})$ by
\begin{equation}\label{eq:MV*}
\mc{M}^*(\mc{V})=\mc{V}\oo\det(\xi)\oo\Sym(\eta^*),
\end{equation}
graded by $\mc{M}^*_r(\mc{V}) = \mc{V}\oo\det(\xi)\oo\Sym^{k-r}(\eta^*)$. Tensoring (\ref{eq:reswedgeeta*}) with $\mc{V}\oo\det(\xi)$, and shifting degrees appropriately (by $-k$), we get an identification in $D^b(X;W)$ between $\mc{M}^*(\mc{V})$ and $\mc{F}^{\bullet}$, where $\mc{F}^{\bullet}$ is the complex resolving $\mc{M}^*(\mc{V})$:
\begin{equation}\label{eq:Fi*}
\mc{F}^i=\mc{V}\oo\bw^{k-i}\xi\oo S^*(i-k),
\end{equation}
with $\mc{F}^i$ living in cohomological degree $i$. Note that in $D^b(X;W)$ we have 
\begin{equation}\label{eq:identificationFs}
\mc{F}^{\bullet}=(\mc{F}_{\bullet}\oo_S S^*)[k],
\end{equation}
where $[k]$ denotes the shift in cohomological degree.

We write $D^b(W)$ for the bounded derived category of graded $S$--modules. The natural projection $\pi:X\times W^*\to W^*$ induces functors $R\pi_{*}:D^b(X;W)\to D^b(W)$, $\pi^*:D^b(W)\to D^b(X;W)$ (since $\pi$ is flat, we don't need to derive $\pi^*$), and we have the projection formula ($\stackrel{L}{\oo}$ denotes the derived tensor product):
\begin{equation}\label{eq:projformula}
R\pi_*(\mc{M}\stackrel{L}{\oo}\pi^*(N))=R\pi_*(\mc{M})\stackrel{L}{\oo} N,\ \rm{for}\ \mc{M}\in D^b(X;W),\ N\in D^b(W). 
\end{equation}
Taking $N$ to be $S^*$ and $\mc{M}=\mc{M}(\mc{V})$, we obtain
\begin{equation}\label{eq:duality}
R\pi_*(\mc{M}^*(\mc{V}))=R\pi_*(\mc{F}^{\bullet})\overset{(\ref{eq:identificationFs})}{=}R\pi_*((\mc{F}_{\bullet}\stackrel{L}{\oo}_S S^*)[k])\overset{(\ref{eq:projformula})}{=}R\pi_*(\mc{F}_{\bullet})\stackrel{L}{\oo}_S S^*[k]=R\pi_*(\mc{M}(\mc{V}))\stackrel{L}{\oo}_S S^*[k]. 
\end{equation}

This yields the following useful method for computing $\Ext$ modules:

\begin{theorem}\label{thm:duality}
 With notation as above, assume that $H^j(X,\mc{M}(\mc{V}))=0$ for $j>0$, and let
 \[M(\mc{V})=H^0(X,\mc{M}(\mc{V})).\] 
 We have for $j\geq 0$ a graded isomorphism
\[\Ext^j_S(M(\mc{V}),S)=H^{k-j}(X,\mc{M}^*(\mc{V}))^*,\]
where $(-)^*$ stands for the graded dual.
\end{theorem}

\begin{proof}
 The hypothesis implies that $M(\mc{V})=R\pi_*(\mc{M}(\mc{V}))$, so (\ref{eq:duality}) yields
\begin{equation}\label{eq:dualityExtH}
H^{k-j}(X,\mc{M}^*(\mc{V}))=\Tor^S_j(M(\mc{V}),S^*)=\Ext_S^j(M(\mc{V}),S)^*, 
\end{equation}
where the last equality follows from \cite[Ex.~3.12]{huneke}.
\end{proof}

\section{Local cohomology with support in maximal minors}\label{sec:maxminors}

In this section we prove the Theorem on Maximal Minors, and derive the Theorem on Covariants of the Special Linear Group. We consider finite dimensional vector spaces $F,G$ with $\dim(F)=m$, $\dim(G)=n$, and assume that $m>n$. We let $W=F\oo G$ and $S=\Sym(W)$. We identify $W^*$ with the vector space of linear maps $G\to F^*$, and think of $S$ as the ring of polynomial functions on this space. We consider the ideal $I$ generated by the maximal minors of the generic matrix in $W^*$ (this generic matrix is written invariantly as the natural map $G\to F^*\oo W$ induced by the identity $G\oo F=W$). More precisely, $I$ is the ideal generated by the (unique) irreducible $\GL(F)\times\GL(G)$--subrepresentation of $S_n=\Sym^n(F\oo G)$ isomorphic to $\bw^n F\oo\bw^n G$ (see~(\ref{eq:cauchy})).

\subsection{The Ext modules}\label{subsec:extminors} Consider the Grassmannian $\bb{G}=\bb{G}(n,F)$ of $n$--dimensional quotients of the vector space $F$, with the tautological exact sequence
\begin{equation}\label{eq:tautologicalGm-nF}
0\lra\mc{R}\lra F\lra\mc{Q}\lra 0, 
\end{equation}
where $\mc{R}$ (resp. $\mc{Q}$) denotes the universal sub-bundle (resp. quotient bundle) of $F$. We have that $\rank(\mc{R})=m-n$ and $\rank(\mc{Q})=n$. Tensoring (\ref{eq:tautologicalGm-nF}) with $G$, we are in the situation of Section~\ref{sec:geomtechnique} with $X=\bb{G}$, $\xi=\mc{R}\oo G$ and $\eta=\mc{Q}\oo G$. We fix $d\geq 0$ and consider the line bundle 
\begin{equation}\label{eq:Vdminors}
\mc{V}=\mc{V}^d=\det(\mc{Q})^{\oo d}\oo\det(G)^{\oo d} 
\end{equation}
on $\bb{G}$, which gives rise to the vector bundle 
\[\mc{M}(\mc{V})=\mc{V}\oo\Sym(\mc{Q}\oo G)=\det(\mc{Q})^{\oo d}\oo\det(G)^{\oo d}\oo\Sym(\mc{Q}\oo G)\]
as in (\ref{eq:MV}).

\begin{lemma}\label{lem:vanishingMVminors}
 With the notation above, $H^j(\bb{G},\mc{M}(\mc{V}))=0$ for $j>0$, and $H^0(\bb{G},\mc{M}(\mc{V}))=I^d(nd)$.
\end{lemma}

\begin{proof} Both $\mc{M}(\mc{V})$ and $\mc{S}=\Sym(\mc{Q}\oo G)$ are $S=\Sym(F\oo G)$--modules, via the natural surjection $S\to\mc{S}$ induced from (\ref{eq:tautologicalGm-nF}). Moreover, we have $\mc{M}(\mc{V})\subset\mc{S}$ is an $S$--submodule, where the inclusion is induced by the inclusion of $\mc{V}$ into $\Sym^{nd}(\mc{Q}\oo\mc{G})$ (see (\ref{eq:cauchy})). If follows that $H^0(\bb{G},\mc{M}(\mc{V}))\subset H^0(\bb{G},\mc{S})=S$ is an ideal. To see that it coincides with $I^{d}$, it suffices to note that its character as a $\GL(F)\times\GL(G)$--representation is the same as that of $I^{d}$ (the shift in degree comes from our grading convention for $\mc{M}(\mc{V})$). This is a consequence of Bott's Theorem~\ref{thm:bott}. The vanishing of the higher cohomology follows from the same theorem.
\end{proof}

\begin{remark}\label{rem:resIdminors}
 Pushing forward (as in \cite[Thm.~5.1.2]{weyman}) the complex (\ref{eq:Fi}) that resolves $\mc{M}(\mc{V})$ we obtain the minimal free resolution of $I^d$, which was originally computed in \cite{akin-buchsbaum-weyman}.
\end{remark}

Lemma~\ref{lem:vanishingMVminors} implies that the hypotheses of Theorem~\ref{thm:duality} apply. Note that $\xi=\mc{R}\oo G$ has rank $n\cdot(m-n)$, so (\ref{eq:dualityExtH}) implies (keeping track of the grading) that
\begin{equation}\label{eq:prelimExtHjminors}
\begin{aligned}
 &\Ext^j_S(I^d,S)_{r-nd} = \Ext^j_S(I^d(nd),S)_{r} = \Ext^j_S(M(\mc{V}),S)_{r}= \\
 &H^{n\cdot(m-n)-j}\left(\bb{G},\Sym^{r+n\cdot(m-n)}(\mc{Q}^*\oo G^*)\oo\det(\mc{R}\oo G)\oo\det(\mc{Q})^{\oo d}\oo\det(G)^{\oo d}\right)^*. 
\end{aligned}
\end{equation}

We get the following description of the modules $\Ext^*_S(I^d,S)$:

\begin{theorem}\label{thm:extIdminors}
 With notation as above, fix $d\geq n$, and define $D_d=\det(F)^{\oo d}\oo\det(G)^{\oo (d+m-n)}$. For $r\in\bb{Z}$ we have
\begin{equation}\label{eq:ExtHjminors}
\Ext^j_S(I^d,S)_r=H^{n\cdot(m-n)-j}\left(\bb{G},\Sym^{r+n\cdot(d+m-n)}(\mc{Q}^*\oo G^*)\oo(\det\mc{R}^*)^{\oo(d-n)}\right)^*\oo D_d^*. 
\end{equation}
If $j$ is not divisible by $(m-n)$, or $j>n\cdot(m-n)$ then $\Ext^j_S(I^d,S)=0$. Otherwise, write $j=s\cdot(m-n)$ for some $0\leq s\leq n$. If $\ll=(\ll_1,\cdots,\ll_n)$ is a dominant weight, we write
\[\ll(s)=(\ll_1,\cdots,\ll_{n-s},\underbrace{-s,\cdots,-s}_{m-n},\ll_{n-s+1}+(m-n),\cdots,\ll_n+(m-n)).\]
We write $W(r;s)$ for the set of dominant weights $\ll=(\ll_1,\cdots,\ll_n)\in\bb{Z}^n$ with $|\ll|=r$ and such that $\ll(s)$ is also dominant. We have
\begin{equation}\label{eq:Extjminors}
\Ext^{s\cdot(m-n)}_S(I^d,S)_r=\bigoplus_{\substack{\ll\in W(r;s) \\ \ll_n\geq -d-(m-n)}} S_{\ll(s)}F\oo S_{\ll}G. 
\end{equation}
\end{theorem}

\begin{proof}
 Equality (\ref{eq:ExtHjminors}) follows from (\ref{eq:prelimExtHjminors}), by replacing $r$ with $r+nd$, and using the relations 
 \[\det(\mc{R}\oo G)=\det(\mc{R})^{\oo n}\oo\det(G)^{\oo(m-n)}\textrm{ and }\det(\mc{Q})^{\oo d}=\det(F)^{\oo d}\oo\det(\mc{R}^*)^{\oo d}.\]
The rest is a direct consequence of Bott's theorem, as we explain next. Writing $t=r+n\cdot(d+m-n)$ and using (\ref{eq:cauchy}) we get
\begin{equation}\label{eq:Hjcauchyminors}
H^j(\bb{G},\Sym^t(\mc{Q}^*\oo G^*)\oo(\det\mc{R}^*)^{\oo(d-n)})=\bigoplus_{\mu\vdash t} H^j(\bb{G},S_{\mu}\mc{Q}^*\oo S_{(d-n)^{m-n}}\mc{R}^*)\oo S_{\mu}G^*. 
\end{equation}
The vanishing of $\Ext^j_S(I^d,S)$ when $s$ is not divisible by $(m-n)$ follows from Remark~\ref{rem:alpharectangle}, so we are left with verifying (\ref{eq:Extjminors}).

Fix $\mu=(\mu_1,\cdots,\mu_n)\vdash t$ and let $\a=(-\mu_n,\cdots,-\mu_1)$ and $\b=((n-d)^{m-n})$. We assume that
\begin{equation}\label{eq:nzHminors}
H^j(\bb{G},S_{\mu}\mc{Q}^*\oo S_{(d-n)^{m-n}}\mc{R}^*)=H^j(\bb{G},S_{\a}\mc{Q}\oo S_{\b}\mc{R})\neq 0.
\end{equation}
We write $\c$ for the concatenation of $\a$ and $\b$, and let $\tl{\c}$ be as in (\ref{eq:tildegamma}). With the convention $\mu_{n+1}=-\infty$, we denote by $s$ the maximal index in $\{0,1,\cdots,n\}$ for which
\[-\mu_{n-s+1}+(m-n)>s-d.\]
Note that if $-\mu_{n-s+1}+(m-n)=s-d$ for some $s$, then $\c+\delta$ has repeated entries, contradicting the assumption (\ref{eq:nzHminors}) (see Theorem~\ref{thm:bott}). We get that the number of pairs $(1\leq x<y\leq m)$ with $\c_x-x<\c_y-y$ is $(n-s)\cdot(m-n)$, and
\[\tl{\c}=(-\mu_n,\cdots,-\mu_{n-s+1},(s-d)^{m-n},-\mu_{n-s}+(m-n),\cdots,-\mu_1+(m-n)).\]
Letting $\ll_i=\mu_i-(d+m-n)$ for $i=1,\cdots,n$, it follows that $\ll$ is a partition of $t-n\cdot(d+m-n)=r$, and for $j=(n-s)\cdot(m-n)$ we get using Theorem~\ref{thm:bott}
\[\left(H^j(\bb{G},S_{\mu}\mc{Q}^*\oo S_{(d-n)^{m-n}}\mc{R}^*)\oo S_{\mu}G^*\right)^*\oo D_d^*=S_{\ll(s)}F\oo S_{\ll}G.\]
It follows from (\ref{eq:ExtHjminors}) and (\ref{eq:Hjcauchyminors}) that all summands of $\Ext^j_S(I^d,S)_r$ have the form $S_{\ll(s)}F\oo S_{\ll}G$ for some $\ll\in W(r;s)$. The condition $\ll_n\geq -d-(m-n)$ in (\ref{eq:Extjminors}) follows from the fact that $\mu_n\geq 0$.

Conversely, assume that $\ll\in\bb{Z}^n$ is a weight occurring in (\ref{eq:Extjminors}). Then $\mu_i=\ll_i+d+(m-n)$ for $i=1,\cdots,n$ defines a partition $\mu\vdash t=r+n\cdot(d+m-n)$. The condition $\ll(s)$ dominant implies
\[\ll_{n-s+1}<\ll_{n-s+1}+(m-n)\leq -s\leq\ll_{n-s},\]
which in turn yields
\[-\mu_{n-s+1}+(m-n)>s-d\textrm{ and }-\mu_{n-s}+(m-n)<s-d+1.\]
It then follows by reversing the steps above that each $S_{\ll(s)}F\oo S_{\ll}G$ is a summand of $\Ext^j_S(I^d,S)_r$, proving~(\ref{eq:Extjminors}).
\end{proof}

\begin{corollary}\label{cor:extIdminors}
 $\Hom_S(I^d,S)=S$, $\Ext^1(S/I^d,S)=0$, and $\Ext^{j+1}_S(S/I^d,S)=\Ext^j_S(I^d,S)$ for $j>0$.
\end{corollary}

\begin{proof}
 The exact sequence $0\lra I^d\lra S\lra S/I^d\lra 0$ induces an exact sequence
\[0=\Hom_S(S/I^d,S)\lra S=\Hom_S(S,S)\overset{\iota}{\lra}\Hom_S(I^d,S)\lra\Ext^1_S(S/I^d,S)\lra\Ext^1_S(S,S)=0.\]
Taking $s=0$ in Theorem~\ref{thm:extIdminors} and noting that the condition of $\ll(0)$ being dominant forces $\ll_n\geq 0$ (and so, in particular, $\lambda_n \geq -d-(m-n)$), we get
\[\Hom_S(I^d,S)_r=\Ext^0_S(I^d,S)_r\overset{(\ref{eq:Extjminors})}{=}\bigoplus_{\ll\vdash r}S_{\ll}F\oo S_{\ll}G\overset{(\ref{eq:cauchy})}{=}\Sym^r(F\oo G),\]
which forces the injective map $\iota$ to be an isomorphism. Since $\iota$ is onto, $\Ext^1_S(S/I^d,S)=0$. Alternatively, the same conclusion follows from the fact that $\textrm{depth}(I^d)=\textrm{depth}(I)=m-n+1>1$, \cite[Thm.~2.5]{bru-vet} and \cite[Cor.~17.8,\ Prop.~18.4]{eisCA}.

The last statement follows from the exact sequence
\[0=\Ext^j_S(S,S)\lra\Ext^j_S(I^d,S)\lra\Ext^{j+1}_S(S/I^d,S)\lra\Ext^{j+1}_S(S,S)=0.\qedhere\]
\end{proof}

\subsection{Local cohomology with support in maximal minors}\label{subsec:loccohminors}

\begin{theorem}\label{thm:loccohminors}
 With the notation as in Theorem~\ref{thm:extIdminors}, the degree $r\in\bb{Z}$ part of the local cohomology modules of $S$ with support in the ideal $I$ can be described as follows
\[H^j_I(S)_r=
\begin{cases}
\bigoplus_{\ll\in W(r;s)} S_{\ll(s)}F\oo S_{\ll}G, & j=s\cdot(m-n)+1,\ s=1,\cdots,n; \\
0, & \textrm{otherwise}.  
\end{cases}
\]
\end{theorem}

\begin{proof}
 Recall the formula (\ref{eq:loccohlimitExt}) for computing local cohomology. The vanishing part of Theorem~\ref{thm:loccohminors} follows from \cite[Thm.~5.10]{witt}, or alternatively by combining (\ref{eq:loccohlimitExt}), Corollary~\ref{cor:extIdminors} and Theorem~\ref{thm:extIdminors}.

Given the description of the $\Ext$ modules from Theorem~\ref{thm:extIdminors} and using Corollary~\ref{cor:extIdminors}, the result will follow once we prove that the successive maps
\begin{equation}\label{eq:Extmapsminors}
\Ext^j_S(I^d,S)\lra\Ext^j_S(I^{d+1},S) 
\end{equation}
in (\ref{eq:loccohlimitExt}) are all injective. Equivalently, we need to show that the natural map in the bounded derived category $D^b(W)$ of graded $S$--modules
\begin{equation}\label{eq:dualExtmapdetminors}
I^{d+1}\stackrel{L}{\oo}_S S^*\to I^d\stackrel{L}{\oo}_S S^*
\end{equation}
induced from the inclusion $I^{d+1}\subset I^d$ yields surjective maps on cohomology.

Recall the definition of the line bundles $\mc{V}^d$ from (\ref{eq:Vdminors}), and denote $\Sym(\mc{Q}\oo G)$ by $\mc{S}$ as before. The natural multiplication map
\begin{equation}\label{eq:multVS}
(\det(\mc{Q})\oo\det(G))\oo\mc{S}\to\mc{S} 
\end{equation}
is a split inclusion of vector bundles on $\bb{G}$. Tensoring it with $\mc{V}^d$ induces a split inclusion $\mc{M}(\mc{V}^{d+1})\hookrightarrow\mc{M}(\mc{V}^d)$. The pushforward of this is precisely the inclusion $I^{d+1}\subset I^d$. The multiplication map (\ref{eq:multVS}) induces an adjoint contraction map
\begin{equation}\label{eq:contractVS*}
(\det(\mc{Q})\oo\det(G))\oo\mc{S}^*\to\mc{S}^*
\end{equation}
which is a split surjection of vector bundles on $\bb{G}$. Tensoring with $\mc{V}^d\oo\det(\mc{R}\oo G)$ we get a split surjection
\begin{equation}\label{eq:maponM*V}
\mc{M}^*(\mc{V}^{d+1})\twoheadrightarrow\mc{M}^*(\mc{V}^d). 
\end{equation}
It is immediate to see that the pushforward of (\ref{eq:maponM*V}) coincides (up to a shift in cohomological degree as in (\ref{eq:duality})) with (\ref{eq:dualExtmapdetminors}). Since (\ref{eq:maponM*V}) is split, it follows that (\ref{eq:dualExtmapdetminors}) induces surjective maps in cohomology, which is what we wanted to prove.
\end{proof}

\subsection{Proof of the Theorem on Covariants of the Special Linear Group}

We are now ready to prove the main application of our calculation of local cohomology:

\begin{theorem}\label{thm:covariants}
 Consider a finite dimensional $\K$--vector space $G$ of dimension $n$, an integer $m>n$, and let $H=\SL(G)$ be the special linear group, $W=G^{\oplus m}$, and $S=\Sym(W)$. If $U=S_{\mu}G$ is the irreducible $H$--representation associated to the partition $\mu=(\mu_1\geq\mu_2\geq\cdots\geq\mu_n=0)$ then $(S\oo U)^H$ is Cohen--Macaulay if and only if $\mu_s-\mu_{s+1}<m-n$ for all $s=1,\cdots,n-1$. 
\end{theorem}

\begin{proof}
We write $W=G^{\oplus m}=F\oo G$ for some vector space $F$ of dimension $m$. Recall that $H=\SL(G)$ denotes the special linear group, $U=S_{\mu}G$ for some partition $\mu=(\mu_1\geq\cdots\geq\mu_n=0)$. Since $S^H$ is a ring of dimension $n\cdot(m-n)+1$, and $(S\oo U)^H$ is a torsion free $S^H$--module, it follows that the Cohen--Macaulayness of $(S\oo U)^H$ is equivalent to the vanishing of the local cohomology modules $H^j_{\m}\left((S\oo U)^H\right)$ for $j<n\cdot(m-n)+1$, where $\m$ denotes the maximal homogeneous ideal of the ring of invariants $S^H$ (note that $\m$ is generated by the maximal minors of the generic $m\times n$ matrix). We have by \cite[Lemma~4.1]{VdB:tracerings} that
\begin{equation}\label{eq:VdBloccohinvariants}
H^j_{\m}\left((S\oo U)^H\right)=\left(H^j_I(S)\oo U\right)^H, 
\end{equation}
where $I\subset S$ is as before the ideal of maximal minors of the generic $m\times n$ matrix. Using (\ref{eq:VdBloccohinvariants}) and Theorem~\ref{thm:loccohminors}, we see that $(S\oo U)^H$ is not Cohen--Macaulay if and only if
\begin{equation}\label{eq:SllSmuHnot0}
(S_{\ll}G\oo S_{\mu}G)^H\neq 0 
\end{equation}
for some $\ll\in W(r;s)$ with $r\in\bb{Z}$, and $1\leq s\leq n-1$. Note that (\ref{eq:SllSmuHnot0}) is equivalent to $S_{\mu}G\simeq S_{\ll}G^*$ as $H$--modules, which in turn is equivalent to
\begin{equation}\label{eq:llcomplmu}
\mu_1+\ll_n=\mu_2+\ll_{n-1}=\cdots=\mu_n+\ll_1. 
\end{equation}
The condition $\ll\in W(r;s)$ for some $r\in\bb{Z}$ is equivalent to $\ll_{n-s}\geq -s\geq\ll_{n-s+1}+(m-n)$, and in particular implies $\ll_{n-s}-\ll_{n-s+1}\geq m-n$. It follows that if (\ref{eq:llcomplmu}) holds for some $\ll\in W(r;s)$ then $\mu_s-\mu_{s+1}=\ll_{n-s}-\ll_{n-s+1}\geq m-n$. We conclude that if $(S\oo U)^H$ is not Cohen--Macaulay then there exists an index $1\leq s\leq n-1$ such that $\mu_s-\mu_{s+1}\geq m-n$.

Conversely, assume that $\mu_s-\mu_{s+1}\geq(m-n)$ for some $1\leq s\leq n-1$. Define $\ll\in\bb{Z}^n$ by letting $\ll_{n-s}=-s$ and $\ll_i=\ll_{n-s}+\mu_{s+1}-\mu_{n-i+1}$ (so that (\ref{eq:llcomplmu}) holds). In particular, we have
\[\ll_{n-s+1}=\ll_{n-s}+\mu_{s+1}-\mu_s=-s-(\mu_s-\mu_{s+1})\leq -s-(m-n),\] 
which shows that $\ll\in W(|\ll|;s)$. By the preceding remarks, this implies that $(S\oo U)^H$ is not Cohen--Macaulay, concluding the proof of the theorem.
\end{proof}

\section{Local cohomology with support in sub--maximal Pfaffians}\label{sec:maxpfaffians}

In this section we prove the Theorem on sub--maximal Pfaffians. We consider a finite dimensional vector space $F$ with $\dim(F)=2n+1$, let $W=\bw^2 F$ and $S=\Sym(W)$. We identify $W^*$ with the vector space of skew--symmetric linear maps $f:F\to F^*$ (i.e. maps satisfying $f^*=-f$), and think of $S$ as the ring of polynomial functions on this space. We consider the ideal $I$ generated by the $2n\times 2n$ Pfaffians of the generic skew--symmetric matrix in $W^*$ (this generic matrix is written invariantly as the natural map $F\to F^*\oo W$ induced by the projection $F\oo F\to W$). More precisely, $I$ is the ideal generated by the (unique) irreducible $\GL(F)$--subrepresentation of $S_n=\Sym^n\left(\bw^2 F\right)$ isomorphic to $\bw^{2n} F$ (see~(\ref{eq:syme2})).


Consider the Grassmannian $\bb{G}=\bb{G}(2n,F)$ of $2n$--dimensional quotients of $F$ (the projective space of lines in $F$) with the tautological sequence (\ref{eq:tautologicalGm-nF}). We have $\rank(\mc{R})=1$ and $\rank(\mc{Q})=2n$. Using \cite[Ex.~5.16]{hartshorne}, (\ref{eq:tautologicalGm-nF}) induces an exact sequence
\[0\lra\mc{R}\oo\mc{Q}\lra\bw^2 F\lra\bw^2\mc{Q}\lra 0.\]
We define $\xi=\mc{R}\oo\mc{Q}$ and $\eta=\bw^2\mc{Q}$, so that the above exact sequence coincides with (\ref{eq:tautologicalV}). We let 
\begin{equation}\label{eq:VdPfaffians}
\mc{V}=\mc{V}^d=\det(\mc{Q})^{\oo d},
\end{equation}
and
\[\mc{M}(\mc{V})=\mc{V}\oo\Sym\left(\bw^2\mc{Q}\right)=\det(\mc{Q})^{\oo d}\oo\Sym\left(\bw^2\mc{Q}\right)\]
as in (\ref{eq:MV}).

\begin{lemma}\label{lem:vanishingMVPfaffians}
With the notation above, $H^j(\bb{G},\mc{M}(\mc{V}))=0$ for $j>0$, and $H^0(\bb{G},\mc{M}(\mc{V}))=I^d(nd)$.
\end{lemma}

\begin{proof} See the proof of Lemma~\ref{lem:vanishingMVminors}.
\end{proof}

\begin{remark}
 As in Remark~\ref{rem:resIdminors}, pushing forward the complex (\ref{eq:Fi}), we obtain the minimal free resolution of $I^d$, which was previously computed in \cite{bof-san}, and in a more general setting by \cite{kustin-ulrich}.
\end{remark}

Noting that $\xi=\mc{R}\oo\mc{Q}$ has rank $2n$, Theorem~\ref{thm:duality} applies to give
\begin{equation}\label{eq:prelimExtHjPfaffians}
\Ext^j_S(I^d,S)_{r-nd}=H^{2n-j}\left(\bb{G},\Sym^{r+2n}\left(\bw^2\mc{Q}^*\right)\oo\det(\mc{R}\oo\mc{Q})\oo\det(\mc{Q})^{\oo d}\right)^*.
\end{equation}

\begin{theorem}\label{thm:extIdpfaffians}
 With notation as above, fix $d\geq 0$, and define $D_d=\det(F)^{\oo (d+1)}$. For $r\in\bb{Z}$ we have
\begin{equation}\label{eq:ExtHjPfaffians}
\Ext^j_S(I^d,S)_r=H^{2n-j}\left(\bb{G},\Sym^{r+n\cdot(d+2)}\left(\bw^2\mc{Q}^*\right)\oo\Sym^{d+1-2n}\mc{R}^*\right)^*\oo D_d^*. 
\end{equation}
If $j$ is odd, or $j>2n$, then $\Ext^j(I^d,S)=0$. Otherwise, write $j=2s$ for some $0\leq s\leq n$. If $\ll=(\ll_1,\cdots,\ll_{2n})$ is a dominant weight, we write
\[\ll(s)=(\ll_1,\cdots,\ll_{2n-2s},-2s,\ll_{2n-2s+1}+1,\cdots,\ll_{2n}+1).\]
We write $W(r;s)$ for the set of dominant weights $\ll=(\ll_1,\cdots,\ll_{2n})\in\bb{Z}^{2n}$ with $|\ll|=r$, all $\ll'_i$ even, and such that $\ll(s)$ is also dominant. We have
\[\Ext^j_S(I^d,S)_r=\bigoplus_{\substack{\ll\in W(r;s) \\ \ll_{2n}\geq -d-2}} S_{\ll(s)}F.\]
\end{theorem}

\begin{proof}
 The equivalence between (\ref{eq:ExtHjPfaffians}) and (\ref{eq:prelimExtHjPfaffians}) follows from the relations
\[\det(\mc{R}\oo\mc{Q})=\det(\mc{R})^{\oo(2n)}\oo\det{\mc{Q}}\textrm{ and }\det(F)=\det(\mc{R})\oo\det(\mc{Q}).\]
The characters of the Ext modules are computed using Bott's theorem, as in the proof of Theorem~\ref{thm:extIdminors}.
\end{proof}

\begin{corollary}\label{cor:extIdpfaffians}
 $\Hom_S(I^d,S)=S$, $\Ext^1(S/I^d,S)=0$, and $\Ext^{j+1}_S(S/I^d,S)=\Ext^j_S(I^d,S)$ for $j>0$.
\end{corollary}

\begin{proof}
 The proof is the same as for Corollary~\ref{cor:extIdminors}.
\end{proof}

With the same proof as Theorem~\ref{thm:loccohminors}, we obtain:

\begin{theorem}\label{thm:loccohpfaffians}
 With the notation as in Theorem~\ref{thm:extIdpfaffians}, the degree $r\in\bb{Z}$ part of the local cohomology modules of $S$ with support in the ideal $I$ can be described as follows
\begin{equation}\label{eq:loccohpfaffians}
H^j_I(S)_r=
\begin{cases}
\bigoplus_{\ll\in W(r;s)} S_{\ll(s)}F, & j=2s+1,\ s=1,\cdots,n; \\
0, & \textrm{otherwise}.  
\end{cases}
\end{equation}
\end{theorem}

\section*{Acknowledgments} 
This work was carried out while we were visiting the Mathematical Sciences Research Institute, for whose hospitality we are grateful. We would like to thank Srikanth Iyengar, Andrew Kustin, Steven Sam, Anurag Singh, Michel Van den Bergh, and Uli Walther for helpful conversations. Experiments with the computer algebra software Macaulay2 \cite{M2} have provided valuable insights. The first author acknowledges the support of the National Science Foundation Grant No.~1303042. The second author acknowledges the support of the Alexander von Humboldt Foundation, and of the National Science Foundation Grant No.~0901185.


\end{document}